\documentclass[11pt,a4paper]{article}

\usepackage[T1]{fontenc}
\usepackage[utf8]{inputenc}
\usepackage{lmodern}
\usepackage[a4paper,margin=28mm,headheight=14pt,headsep=7mm,footskip=11mm]{geometry}
\usepackage{amsmath,amssymb,amsthm,mathtools}
\usepackage{microtype}
\usepackage{fancyhdr}
\usepackage{needspace}
\usepackage{aliascnt}
\usepackage[hidelinks,pdfencoding=auto,psdextra]{hyperref}
\usepackage[nameinlink,noabbrev,capitalise]{cleveref}

\hypersetup{
  pdftitle={Sudakov minoration for unconditional log-concave vectors with negatively associated magnitudes},
  pdfauthor={Witold Bednorz, Rafa\l{} Martynek and Rafa\l{} Meller},
  pdfsubject={Sudakov minoration, common witnesses and negative association},
  pdfkeywords={Sudakov minoration, log-concave measure, negative association, common witnesses, Bernoulli process}
}

\allowdisplaybreaks[1]
\numberwithin{equation}{section}

\newtheorem{theorem}{Theorem}[section]
\newaliascnt{lemma}{theorem}
\newtheorem{lemma}[lemma]{Lemma}
\aliascntresetthe{lemma}
\newaliascnt{proposition}{theorem}
\newtheorem{proposition}[proposition]{Proposition}
\aliascntresetthe{proposition}
\newaliascnt{corollary}{theorem}
\newtheorem{corollary}[corollary]{Corollary}
\aliascntresetthe{corollary}
\theoremstyle{definition}
\newtheorem*{definition}{Definition}
\theoremstyle{remark}
\newaliascnt{remark}{theorem}
\newtheorem{remark}[remark]{Remark}
\aliascntresetthe{remark}
\crefname{lemma}{Lemma}{Lemmas}
\crefname{proposition}{Proposition}{Propositions}
\crefname{corollary}{Corollary}{Corollaries}
\crefname{remark}{Remark}{Remarks}

\newcommand{\E}{\mathbb E}
\newcommand{\Prob}{\mathbb P}
\newcommand{\R}{\mathbb R}
\newcommand{\one}{\mathbf 1}
\newcommand{\supp}{\operatorname{supp}}
\newcommand{\conv}{\operatorname{conv}}
\newcommand{\ip}[2]{\langle #1,#2\rangle}
\newcommand{\norm}[1]{\lVert #1\rVert}
\newcommand{\abs}[1]{\lvert #1\rvert}
\newcommand{\W}{\mathcal W}
\newcommand{\calC}{\mathcal C}
\newcommand{\calK}{\mathcal K}
\newcommand{\calN}{\mathcal N}

\title{\textbf{Sudakov minoration for unconditional\\
log-concave vectors with negatively\\
associated magnitudes}}
\author{Witold Bednorz\qquad Rafa\l{} Martynek\qquad Rafa\l{} Meller\\[1ex]
  \small Faculty of Mathematics, Informatics and Mechanics\\
  \small University of Warsaw\\
  \small\href{mailto:wbednorz@mimuw.edu.pl}{\texttt{wbednorz@mimuw.edu.pl}}}
\date{}

\begin{document}
\maketitle
\thispagestyle{plain}

\begin{abstract}
We prove the Sudakov minoration principle, with a universal constant,
for unconditional log-concave random vectors whose coordinate magnitudes
are negatively associated. The main step is a deterministic selection
theorem: pairwise witnesses in convex downward-closed sets yield, on an
exponentially large subfamily, one fixed witness per label separating
it from every other retained label. The selection combines a prefix
counting inequality with convex separation in the space of all
label-coordinate arrays. Negative association then controls the
expected number of active joint-tail witnesses, while a Bernoulli
argument on complementary coordinate sets retains the random signs.
The proof uses joint upper-orthant probabilities and applies negative
association only under the original law. Consequences include minoration
for products of Orlicz-based log-concave measures and a covering estimate
in the moment metric, expressed geometrically through polar $L_p$
centroid bodies.
\end{abstract}

\begingroup\small
\par\endgroup

\section{Introduction}\label{sec:introduction}

For a symmetric random vector $X$ in $\R^d$ and a finite nonempty set
$T\subset\R^d$, write
\[
 \W_X(T)=\E\max_{t\in T}\ip{t}{X},\qquad
 d_p(s,t)=\norm{\ip{t-s}{X}}_p\quad(p\ge1).
\]
Here $\norm{V}_p=(\E|V|^p)^{1/p}$, and all logarithms are natural.
The Sudakov minoration principle asks whether the conditions
$|T|\ge e^p$ and $d_p(s,t)\ge a$ for distinct $s,t\in T$ imply
$\W_X(T)\ge ca$, with a constant independent of $d$, $p$, and $T$.

For Gaussian processes this is the classical entropy lower bound of
Sudakov~\cite{Sudakov}. Talagrand established its Bernoulli counterpart
\cite{TalagrandBernoulli} and obtained minoration for important classes
of canonical processes~\cite{TalagrandCanonical}. Lata\l a~\cite{Latala1997}
treated independent symmetric coordinates with log-concave tails. The
role of regular moment growth in the independent-coordinate setting
was subsequently studied by Lata\l a and Tkocz~\cite{LatalaTkocz}.

The dependent log-concave setting leads to a broader question.
Lata\l a~\cite{LatalaSMP} formulated the conjecture that all log-concave
vectors satisfy Sudakov minoration with a universal constant, proved
several special cases, and established weaker minoration statements
in general. For unconditional vectors, the joint-tail description of
moments~\cite{LatalaMoments} and the sparse support reduction
in~\cite{Bednorz} turn the problem into one of organizing witnesses
on support differences. The common-witness approach was developed
in~\cite[Section~6]{Bednorz}. Minoration for products of radial-type
log-concave measures was studied in~\cite{BednorzRadial}.

Negative association provides a natural dependence assumption for this
problem. In the form introduced by Joag-Dev and Proschan~\cite{JoagDev},
it bounds the covariance of increasing functions of disjoint coordinate
sets. Pilipczuk and Wojtaszczyk~\cite{PilipczukWojtaszczyk} proved
negative association of the coordinate magnitudes for uniform measures
on generalized Orlicz balls. Wojtaszczyk~\cite{Wojtaszczyk} gave a
simpler proof based on localization and extended the class of admissible
measures. Here we assume negative association of the magnitudes and
use it directly to control overlapping witness events.

\subsection{The main result}

A law is \emph{unconditional} if it is invariant under changing any
collection of coordinate signs. We write $[d]=\{1,\ldots,d\}$.

\begin{definition}
A random vector $Y=(Y_i)_{i=1}^d$ has \emph{negatively associated}
(NA) coordinates if, for disjoint $I,J\subset[d]$ and bounded
coordinatewise increasing functions $f$ and $g$,
\[
 \E f(Y_I)g(Y_J)\le \E f(Y_I)\,\E g(Y_J).
\]
For nonnegative increasing functions the same inequality follows by
truncation whenever the expectations involved are finite.
\end{definition}

\begin{theorem}[Sudakov minoration under NA of magnitudes]\label{thm:main}
There exists a universal constant $c>0$ with the following property.
Let $X$ be a unconditional log-concave random vector in $\R^d$ and
assume that $(|X_i|)_{i=1}^d$ has negatively associated coordinates.
For every $P\ge1$, every finite $T\subset\R^d$ with $|T|\ge e^P$, and
every $a>0$ such that
\begin{equation}\label{eq:original-packing}
  \norm{\ip{t-s}{X}}_P\ge a
       \qquad(s,t\in T,\ s\ne t),
\end{equation}
one has
\begin{equation}\label{eq:main-smp}
  \W_X(T)\ge ca.
\end{equation}
The constant is independent of the dimension, the law, $P$, and $T$.
\end{theorem}

The NA assumption concerns the \emph{magnitudes}; no assertion about
arbitrary negatively associated signed vectors is intended.
Isotropy is not an additional hypothesis. Unconditionality makes the
covariance diagonal, so one may delete zero coordinates and apply positive
diagonal scaling when using isotropic forms of the classical inputs.
Such operations preserve NA of the magnitudes.

\subsection{The proof mechanism}

The central reduced process is
\begin{equation}\label{eq:binary-process}
  Z_t=\sum_{i\in A_t}\varepsilon_iY_i,
  \qquad Y_i\ge0,\quad |A_t|\le m,
\end{equation}
where the $\varepsilon_i$ are independent symmetric signs, independent of
$Y$. The support size $m$ is a small fraction of $\log|T|$.
A moment separation gives, for each pair, a feasible witness on one of
the directed differences $A_t\setminus A_s$ and $A_s\setminus A_t$.
The first problem is to choose witnesses simultaneously.

The encoder construction solves that problem using only convexity and downward closure
of the feasible witness bodies. It gives a retained set $T'$ and vectors
$a_t$, fixed once and for all, with both a controlled joint tail cost and
\[
 a_t(A_t\setminus A_s)\ge c u
       \qquad(t\ne s\text{ in }T').
\]
Here and below $v(D)=\sum_{i\in D}v_i$ for a nonnegative vector $v$.
The vectors $a_t$ may depend on $t$; the result does not assert the
existence of one coordinatewise threshold vector shared by all labels.

The second problem is to control overlapping witness events. If
$B_t=\{Y\ge b_t\}$ and $\Lambda=\sum_t\Prob(B_t)$, the signed counting
lemma below proves
\begin{equation}\label{eq:roadmap-count}
 \E\max_t|Z_t|
 \ge c_B\beta\left(1-\frac{2^m\lceil e^{2m}\rceil}{\Lambda}\right)_+,
\end{equation}
provided $b_t(A_t\setminus A_s)\ge\beta$ for all distinct ordered pairs.
Log-concavity makes $\Lambda$ exponentially large after a fixed shrinking
of the selected witnesses. The small support size makes the numerator in
\eqref{eq:roadmap-count} at most $\Lambda/2$.

\Cref{thm:encoder} gives the deterministic selection, and
\cref{lem:signed-count} gives the signed counting estimate under NA.
The sparse reduction connects these statements to \cref{thm:main}.
The two applications in \cref{sec:consequences} concern Orlicz-based
measures and coverings in the $L_p$ moment metric.

\section{Preliminaries}\label{sec:preliminaries}

\subsection{Signs, upper orthants, and convex witness bodies}

An unconditional vector has the representation in distribution
\begin{equation}\label{eq:sign-representation}
  X\overset{d}= (\varepsilon_iR_i)_{i=1}^d,
  \qquad R=(|X_i|)_{i=1}^d,
\end{equation}
with independent signs independent of $R$. For a nondegenerate
unconditional log-concave density $f$, the density of $R$ on $\R_+^d$
is $2^df$ restricted to that orthant, and is log-concave. The same statement
holds on the coordinate subspace supporting a degenerate unconditional
law. Diagonal coordinate scalings preserve log-concavity; they
also preserve NA when it is assumed.

\begin{lemma}[Joint-tail convexity and shrinking]\label{lem:tail}
Let $Y$ be a log-concave random vector supported on $\R_+^d$, and set
\[
 Q_Y(v)=\Prob(Y\ge v),\qquad v\in\R_+^d,
\]
where inequalities between vectors are coordinatewise. The function
$Q_Y$ is log-concave and upper semicontinuous. In particular,
\begin{equation}\label{eq:tail-shrink}
 Q_Y(\theta v)\ge Q_Y(v)^\theta
       \qquad(0<\theta\le1).
\end{equation}
For $p,u>0$ and $A\subset[d]$, the set
\begin{equation}\label{eq:actual-body}
 K_A(p,u)=\left\{v\ge0:\ \supp v\subset A,\quad
           Q_Y(v)\ge e^{-p},\quad \norm{v}_1\le u\right\}
\end{equation}
is nonempty, compact, convex, and downward closed in the positive orthant.
\end{lemma}

\begin{proof}
For $v,w\ge0$ and $0<\theta<1$,
\[
 \theta(v+\R_+^d)+(1-\theta)(w+\R_+^d)
   =\theta v+(1-\theta)w+\R_+^d.
\]
Log-concavity of the measure gives log-concavity of $Q_Y$; unbounded
orthants can be treated by compact approximation. Since $Q_Y(0)=1$,
this also proves \eqref{eq:tail-shrink}.
If $v_n\to v$, then
$\limsup_n\one_{\{Y\ge v_n\}}\le\one_{\{Y\ge v\}}$ pointwise.
Reverse Fatou proves upper semicontinuity. The properties of
\eqref{eq:actual-body} now follow from its definition: the tail superlevel
set is closed and convex, decreasing thresholds can only increase its
probability, and the $\ell_1$ cap makes it bounded.
\end{proof}

\subsection{Inputs from the existing theory}

The proof uses the following standard results in the precise forms stated
here. Their proofs are not repeated.

\begin{proposition}[Bernoulli Sudakov minoration]\label{input:bernoulli}
There is a universal $c_B>0$ such that, if $q\ge1$, a finite
$V\subset\R^d$ satisfies $|V|\ge e^q$, and
\[
 \norm{\sum_i(v_i-w_i)\varepsilon_i}_q\ge b
          \qquad(v\ne w\text{ in }V),
\]
then
\[
 \E_\varepsilon\max_{v\in V}\left|\sum_i v_i\varepsilon_i\right|
       \ge c_B b.
\]
\end{proposition}

This is Talagrand's Bernoulli minoration theorem
\cite{TalagrandBernoulli}; see also \cite[Example~1.4]{LatalaSMP}.

\begin{proposition}[A joint-tail witness for a sparse linear form]\label{input:witness}
There is a universal $C_w\ge1$ such that the following holds.
Let $X$ be unconditional and log-concave, $P\ge1$, $|J|\le P$,
and $r_i\ge0$. There is $z\in\R_+^J$ with
\begin{equation}\label{eq:classical-witness}
 \Prob\bigl(|X_i|\ge z_i\text{ for all }i\in J\bigr)\ge e^{-P},
 \qquad
 \sum_{i\in J}r_i z_i
 \ge C_w^{-1}\norm{\sum_{i\in J}r_iX_i}_P.
\end{equation}
\end{proposition}

This is the sparse-support case of the moment characterization in
\cite{LatalaMoments}; it is also stated in \cite[Theorem 1]{Bednorz}.

\begin{proposition}[Sparse binary reduction]\label{input:reduction}
Fix $\eta,\epsilon>0$. There are constants
$c_{\eta,\epsilon}>0$ and $P_0(\eta,\epsilon)<\infty$ such that, for
$P\ge P_0(\eta,\epsilon)$ and an unconditional log-concave $X$ satisfying
\eqref{eq:original-packing}, either
\begin{equation}\label{eq:easy-reduction-branch}
 \W_X(T)\ge c_{\eta,\epsilon}a,
\end{equation}
or there are $n=\lceil e^{P/4}\rceil$ distinct sets $A_t\subset[d]$
and weights $r_i\ge0$ such that
\begin{align}
 |A_t|&\le\eta P,\label{eq:reduction-support}\\
 \norm{\sum_i r_i(\one_{A_t}(i)-\one_{A_s}(i))X_i}_P
       &\ge a/2\qquad(s\ne t),\label{eq:reduction-packing}\\
 M_{\rm bin}:=\E\max_t\left|\sum_{i\in A_t}r_iX_i\right|
       &\le2\W_X(T)+\epsilon a.\label{eq:reduction-transfer}
\end{align}
\end{proposition}

The construction changes labels and coefficients, not the law of the
magnitudes. This is the quantitative reduction underlying
\cite[Proposition 2 and Corollary 2]{Bednorz}.

For clarity, the parameters in this restatement can be recovered from
the proof of \cite[Proposition~2]{Bednorz}. After rescaling the labels
so that the original separation is $P$, that proof either gives
$\W_X(T)\ge c_\delta P$ or produces a translated subfamily and an
approximation $\varphi(t)$ with common nonzero coordinate values and
\[
 |\supp\varphi(t)|\le\frac{P}{4\log(1/\rho)},\qquad
 \norm{\ip{t-\varphi(t)}{X}}_P\le C(\rho+\delta)P,
\]
where $\rho/\log(1/\rho)=C'\delta$ and the constants are universal.
Choose $\rho$ sufficiently small that the support bound is at most
$\eta P$ and the moment error is at most $P/4$. For a retained family
of size $n=\lceil e^{P/4}\rceil$, the estimate
\[
 \E\max_t|V_t|\le n^{1/P}\max_t\norm{V_t}_P
\]
makes the approximation error at most $\epsilon P$ after a further
decrease of $\rho$. These choices depend only on $\eta$ and
$\epsilon$. The triangle inequality gives
\eqref{eq:reduction-packing}. Translation by an anchor, followed by
symmetry and centering, bounds the expected absolute supremum of the
translated subfamily by $2\W_X(T)$. Rescaling back gives
\eqref{eq:reduction-transfer}. Thus both parameters are fixed before
the reduction is applied.

For bounded values of $P$ we also use the usual scalar log-concave moment
comparison: for a centered log-concave real random variable $V$,
\begin{equation}\label{eq:scalar-growth}
 \norm{V}_P\le C_{\rm reg}P\E|V|\qquad(P\ge1),
\end{equation}
with a universal constant $C_{\rm reg}$; see, for instance,
\cite[Section 2]{LatalaSMP}.

\section{Prefix encoders of distinct sub-supports}\label{sec:encoders}

Fix a finite label set $T$, with $N=|T|\ge2$, and supports
$A_t\subset[d]$ of size at most an integer $m$. Put $p=\log N$.
Fix once and for all an ordering of the ground coordinates.

An \emph{admissible encoder} chooses a complete family of pairwise
distinct sub-supports
\begin{equation}\label{eq:cores}
   C_t\subset A_t\qquad(t\in T).
\end{equation}
The coordinate order need not vary between encoders. The whole family
in \eqref{eq:cores} is part of the encoder; no separate choice of an
encoder is made for each label.

For a fixed encoder $E$, partition labels by $k=|C_t|$. Let $N_k$
be the number in that class, and choose a label uniformly in it.
Write its ordered sub-support as $(i_1,\ldots,i_k)$. If $n_j(t)$ is
the number of labels in the class sharing the first $j$ coordinates of
this word, then $n_0(t)=N_k$ and $n_k(t)=1$. Define
\begin{equation}\label{eq:tau-definition}
 \tau^E_{i_j}(t)=\log\frac{n_{j-1}(t)}{n_j(t)}\quad(1\le j\le k),
 \qquad \tau^E_i(t)=0\quad(i\notin C_t).
\end{equation}
These are the negative logarithms of the successive conditional
probabilities of the prefix code. In particular,
\begin{equation}\label{eq:tau-budget}
 \tau_t^E\ge0,\qquad \supp\tau_t^E\subset C_t,
 \qquad \sum_i\tau_i^E(t)=\log N_k\le p.
\end{equation}
If $k=0$, distinctness gives $N_0=1$ and the information vector is zero.

\begin{lemma}[Counting against any reference support]\label{lem:kraft}
For every admissible encoder and every $S\subset[d]$ with $|S|\le m$,
\begin{equation}\label{eq:kraft}
 \sum_{t\in T}\exp\bigl(-\tau_t^E(A_t\setminus S)\bigr)
       \le B_m,\qquad B_m=(m+1)4^m.
\end{equation}
In particular, $S$ need not be one of the encoded sub-supports.
\end{lemma}

\begin{proof}
Fix a length $k$. Further classify an ordered word $(i_1,\ldots,i_k)$
by the positions $J\subset[k]$ at which its coordinates belong to $S$,
and by the set $D\subset S$ of those coordinates. Necessarily
$|J|=|D|$. The ordering fixes which member of $D$ occupies each position
in $J$.

For one fixed pattern $(k,J,D)$, run the following probability
experiment. At a position in $J$, force the prescribed coordinate of $D$.
At any other position, sample according to the original prefix
conditional probabilities in the size-$k$ class. On a prefix absent
from the original tree, define arbitrary transition probabilities;
such a prefix cannot be extended into a valid word of that tree.

If a valid label $t$ has this pattern, the probability that this
experiment produces its word is
\[
 \prod_{j\notin J}\frac{n_j(t)}{n_{j-1}(t)}
   =\exp\bigl(-\tau_t^E(C_t\setminus S)\bigr).
\]
The accepted words are distinct, so their probabilities sum to at most
one. There are at most $2^k2^{|S|}$ choices of $(J,D)$ for each $k$,
and at most $m+1$ possible lengths. Summing over all patterns gives
$(m+1)4^m$. Finally, $\tau_t^E$ vanishes off $C_t\subset A_t$, so
$\tau_t^E(C_t\setminus S)=\tau_t^E(A_t\setminus S)$.
\end{proof}

\begin{remark}
The reference set in \cref{lem:kraft} may be the original $A_s$ even
though the encoder uses a smaller $C_s$. This is what allows different
complete sub-support families to be averaged while retaining one fixed
counting inequality. The factor $m+1$ accounts for the possible size classes.
\end{remark}

\section{Simultaneous selection of witnesses}\label{sec:selection}

The next result is deterministic. In particular, it does not use NA.
For a compact set $K\subset\R_+^d$, write
\[
 h_K(w)=\max_{v\in K}\ip{w}{v}.
\]
Downward closed means that $0\le z\le v\in K$ implies $z\in K$.

\begin{theorem}[Common distribution of encoders]\label{thm:encoder}
Let $N=|T|\ge2$, $p=\log N$, and $|A_t|\le m$.
Let $K_t\subset\R_+^d$ be nonempty compact convex downward-closed
sets supported on $A_t$. Suppose $u>0$ and
\begin{equation*}
 \max\left\{
 h_{K_t}(\one_{A_t\setminus A_s}),
 h_{K_s}(\one_{A_s\setminus A_t})
 \right\}\ge u\qquad(s\ne t).
 \tag{H}\label{eq:H}
\end{equation*}
There is one probability law $\nu$ on admissible encoders such that
\begin{equation}\label{eq:common-law}
 a_t:=\frac{u}{2p}\E_\nu\tau_t^E\in K_t
                  \qquad\text{for every }t\in T.
\end{equation}
These vectors satisfy, for every $s\in T$,
\begin{equation}\label{eq:exponential-selection}
 \sum_{t\in T}
  \exp\left(-\frac{2p}{u}a_t(A_t\setminus A_s)\right)
       \le B_m.
\end{equation}
\end{theorem}

\begin{proof}
Condition \eqref{eq:H} forces the original $A_t$ to be distinct, so
the choice $C_t=A_t$ is admissible. There are finitely many admissible
families of sub-supports. Let
\[
 \calC=\conv\left\{(\tau_t^E)_{t\in T}:E\text{ admissible}\right\},
 \qquad \calK=\prod_{t\in T}K_t.
\]
Both are compact convex sets in the finite-dimensional space of
label-coordinate arrays. We prove
\begin{equation}\label{eq:intersection}
       \frac{u}{2p}\calC\cap\calK\ne\varnothing.
\end{equation}

Fix nonnegative costs $w_t$ on $A_t$, and put
\begin{equation}\label{eq:cheap-cores}
 H_t=h_{K_t}(w_t),\qquad
 C_t=\left\{i\in A_t:w_{t,i}\le\frac{2H_t}{u}\right\}.
\end{equation}
We first check that the $C_t$ are distinct. For a pair $s\ne t$,
condition \eqref{eq:H}, downward closure, and scaling down give either
a vector $v\in K_t$ of mass exactly $u$ supported on
$A_t\setminus A_s$, or the corresponding vector at endpoint $s$.
Consider the first case. Since $\ip{w_t}{v}\le H_t$, the mass of $v$
outside $C_t$ is at most $u/2$ when $H_t>0$. If $H_t=0$, that mass
is zero. Hence
\[
   v(C_t\setminus A_s)\ge u/2>0.
\]
Thus $C_t\not\subset A_s$, whereas $C_s\subset A_s$, and $C_t\ne C_s$.
The other endpoint gives the same conclusion.

The family in \eqref{eq:cheap-cores} therefore defines an admissible
encoder $E$. Every coordinate used by label $t$ has cost at most
$2H_t/u$. By \eqref{eq:tau-budget},
\[
 \ip{w_t}{\tau_t^E}\le\frac{2H_t}{u}\sum_i\tau_i^E(t)
                         \le\frac{2p}{u}H_t.
\]
It follows that, for every nonnegative cost array,
\begin{equation}\label{eq:dual-certificate}
 \frac{u}{2p}\min_{\tau\in\calC}
     \sum_t\ip{w_t}{\tau_t}
       \le\sum_t h_{K_t}(w_t)=h_{\calK}(w).
\end{equation}

If \eqref{eq:intersection} failed, strict separation of the two compact
convex sets would give a linear functional $w$ with
\[
 \inf_{z\in(u/(2p))\calC}\ip{w}{z}>h_{\calK}(w).
\]
All encoder arrays are nonnegative. Write $w_+$ for the
coordinatewise positive part of $w$. Downward closure of $\calK$ gives
$h_{\calK}(w)=h_{\calK}(w_+)$: an optimizer for $w_+$ may have all
coordinates with negative $w$ set to zero. Replacing $w$ by $w_+$ can
only increase the infimum on the left. This contradicts
\eqref{eq:dual-certificate} and proves \eqref{eq:intersection}.

A point of $\calC$ is the average of whole encoder arrays under one
finite probability law $\nu$. Thus \eqref{eq:intersection} is precisely
\eqref{eq:common-law}. Finally, Jensen's inequality and
\cref{lem:kraft}, used with $S=A_s$, give
\begin{align*}
 \sum_t e^{-(2p/u)a_t(A_t\setminus A_s)}
 &=\sum_t e^{-\E_\nu\tau_t^E(A_t\setminus A_s)}\\
 &\le\E_\nu\sum_t e^{-\tau_t^E(A_t\setminus A_s)}
 \le B_m.
\end{align*}
\end{proof}

\begin{remark}
The cost array in the separation proof may select a different complete
family $(C_t)_t$. The conclusion is nevertheless one common mixing law
on complete families, because the separation takes place in the product
space of \emph{all} labels. Applying a separate minimax argument for
each label would not give this conclusion.
\end{remark}

\begin{corollary}[An exponential subfamily with fixed witnesses]\label{cor:selection}
Under the hypotheses of \cref{thm:encoder}, there is $T'\subset T$ with
\begin{equation}\label{eq:selection-size}
 |T'|\ge\frac{e^p}{1+2B_m e^{p/4}}
\end{equation}
such that the vectors in \eqref{eq:common-law} satisfy
\begin{equation}\label{eq:ordered-gains}
 a_t(A_t\setminus A_s)\ge u/8
             \qquad(t\ne s\text{ in }T').
\end{equation}
If $m\ge1$ and $m\le p/64$, then $|T'|\ge e^{p/2}$.
\end{corollary}

\begin{proof}
For a fixed $s$, \eqref{eq:exponential-selection} shows that at most
$D=B_me^{p/4}$ labels $t$ violate the inequality in
\eqref{eq:ordered-gains}. Join two distinct labels by an undirected
edge if either ordered inequality fails. This graph has at most $ND$
edges, so its average degree is at most $2D$. It has an independent set
of size at least $N/(1+2D)$. For completeness, a random ordering of
vertices retains those preceding all their neighbors, giving expected
size $\sum_t(1+\deg t)^{-1}\ge N/(1+2D)$.

For the last assertion, $m+1\le2^m$ gives $B_m\le8^m$.
Since $D\ge1$ and $p\ge64$,
\[
 |T'|\ge\frac13\exp\left(\frac{3p}{4}-m\log8\right)
 \ge\frac13\exp\left(\left(\frac34-\frac{\log8}{64}\right)p\right)
 \ge e^{p/2}.
\]
\end{proof}

For the bodies $K_t=K_{A_t}(p,u)$ from \eqref{eq:actual-body}, membership
$a_t\in K_t$ is the actual probability bound
$\Prob(Y\ge a_t)\ge e^{-p}$. No product of marginal probabilities
has replaced this cost at any stage.

\section{Negative association and active witness counts}\label{sec:counting}

Throughout this section $Y\ge0$ is NA, the family $T$ is finite with
$|T|\ge2$, and
the deterministic vector $b_t\ge0$ is supported on $A_t$. Assume
\begin{equation}\label{eq:b-separation}
 b_t(A_t\setminus A_s)\ge\beta>0
                      \qquad(t\ne s).
\end{equation}
Define
\begin{equation}\label{eq:active-count}
 B_t=\{Y\ge b_t\},\qquad q_t=\Prob(B_t),\qquad
 \calN(Y)=\sum_t\one_{B_t},\qquad \Lambda=\sum_t q_t.
\end{equation}
We assume $\Lambda>0$ and write $x_+=\max\{x,0\}$.
Log-concavity is not needed for the counting results in this section.

\subsection{The positive counting estimate}

\begin{proposition}\label{prop:positive-count}
Let $F(Y)=\max_t\sum_{i\in A_t}Y_i$. Then
\begin{equation}\label{eq:positive-count}
 \E F(Y)\ge\beta(1-\Lambda^{-1})_+.
\end{equation}
\end{proposition}

\begin{proof}
For a fixed $t$, put
\[
 F_t(Y)=\max_{s\ne t}\sum_{i\in A_s\setminus A_t}Y_i.
\]
This is increasing and depends only on $A_t^c=[d]\setminus A_t$,
while $B_t$ is
increasing and depends only on $A_t$. Also $F_t\le F$. Thus NA gives
\[
 \E[\one_{B_t}F_t]\le q_t\E F_t\le q_t\E F.
\]
If $B_t$ occurs and $\calN\ge2$, some other $B_s$ occurs, and
$F_t\ge b_s(A_s\setminus A_t)\ge\beta$. Summing in $t$,
\[
 \Lambda\E F\ge\beta\E[\calN\one_{\{\calN\ge2\}}]
   =\beta\bigl(\Lambda-\Prob(\calN=1)\bigr)
   \ge\beta(\Lambda-1).
\]
The assertion follows, with the case $\E F=\infty$ understood trivially.
\end{proof}

\subsection{The signed counting estimate}

\begin{lemma}[Signed counting under NA]\label{lem:signed-count}
Assume \eqref{eq:b-separation}, $|A_t|\le m$, and $m\ge1$. Let
\[
 M=\E_{Y,\varepsilon}\max_t
       \left|\sum_{i\in A_t}\varepsilon_iY_i\right|,
 \qquad L_m=2^m\lceil e^{2m}\rceil,
\]
where the signs are independent of $Y$. With the constant $c_B$ in
\cref{input:bernoulli},
\begin{equation}\label{eq:signed-count}
 M\ge c_B\beta\left(1-\frac{L_m}{\Lambda}\right)_+.
\end{equation}
\end{lemma}

\begin{proof}
We may assume $M<\infty$. For each fixed reference label $t$, define
\begin{equation}\label{eq:G-definition}
 G_t(y)=\E_\varepsilon\max_{s\in T}
       \left|\sum_{i\in A_s\setminus A_t}\varepsilon_i y_i\right|.
\end{equation}
As a function of one coordinate, this expectation is convex and even,
and hence nondecreasing on $[0,\infty)$. Thus $G_t$ is increasing on
the positive orthant and depends only on coordinates outside $A_t$.
Moreover, averaging first over the signs in $A_t$ and using convexity
of the maximum of absolute values shows, for every $y\ge0$, that
\[
 G_t(y)\le\E_\varepsilon\max_s
               \left|\sum_{i\in A_s}\varepsilon_i y_i\right|.
\]
Consequently, NA gives
\begin{equation}\label{eq:NA-G}
 \E[\one_{B_t}G_t(Y)]\le q_t\E G_t(Y)\le q_tM.
\end{equation}
For an unbounded $G_t$, apply NA to $G_t\wedge R$ and let $R\to\infty$.

We claim that
\begin{equation}\label{eq:many-active}
 \calN(y)>L_m\quad\Longrightarrow\quad
       G_t(y)\ge c_B\beta\quad\text{for every }t\in T.
\end{equation}
Fix such $y$ and a reference $t$. Partition the active labels
$\{s:y\in B_s\}$ according to $A_s\cap A_t$. There are at most
$2^{|A_t|}\le2^m$ classes. Some class $\mathcal S$ therefore has
more than $e^{2m}$ labels.

For $s\in\mathcal S$ define the deterministic vector
$v_s=y\one_{A_s\setminus A_t}$. If $r,s\in\mathcal S$ are distinct,
their supports agree inside $A_t$, so
\begin{align}
 \norm{v_s-v_r}_1
 &=y(A_s\setminus A_r)+y(A_r\setminus A_s)\notag\\
 &\ge b_s(A_s\setminus A_r)+b_r(A_r\setminus A_s)
 \ge2\beta.\label{eq:active-l1}
\end{align}
The difference has at most $2m$ nonzero coordinates. The event on which
every relevant Rademacher sign agrees with its coefficient therefore
gives
\begin{equation}\label{eq:bernoulli-sparse-moment}
 \norm{\sum_i(v_s-v_r)_i\varepsilon_i}_{2m}
 \ge2^{-\abs{\supp(v_s-v_r)}/(2m)}\norm{v_s-v_r}_1
 \ge\beta.
\end{equation}
The vectors are distinct by \eqref{eq:active-l1}, and
$|\mathcal S|>e^{2m}$. Applying \cref{input:bernoulli} at $q=2m$
proves \eqref{eq:many-active}.

Now sum \eqref{eq:NA-G}. By \eqref{eq:many-active},
\begin{align}
 \Lambda=\E\calN
 &\le L_m+\E[\calN\one_{\{\calN>L_m\}}]\notag\\
 &\le L_m+\frac1{c_B\beta}\sum_t
                 \E[\one_{B_t}G_t(Y)]\notag\\
 &\le L_m+\frac{M}{c_B\beta}\Lambda.\label{eq:count-master}
\end{align}
Rearrangement gives \eqref{eq:signed-count}.
\end{proof}

\begin{remark}[The two fixed coordinate blocks]
The function $G_t$ in \eqref{eq:G-definition} uses the entire fixed
label family. It is not a function chosen after conditioning on the
active labels. The active class $\mathcal S$ is used only to prove a
pointwise lower bound for that fixed function. Therefore the sole NA
application in \eqref{eq:NA-G} is between $A_t$ and $A_t^c$ under the
original law. Neither conditional NA nor an expected-supremum comparison
deduced from individual linear forms is being assumed.
\end{remark}

\section{The reduced signed minoration}\label{sec:reduced}

\begin{theorem}[From directed witnesses to signed minoration]\label{thm:reduced}
Let $Y\ge0$ be log-concave and NA. Let $T$ be a finite label set,
$N=|T|\ge2$, and $p=\log N$. Suppose the supports satisfy
$|A_t|\le m\le p/64$. For some $u>0$, form
\[
 K_t=K_{A_t}(p,u)
\]
as in \eqref{eq:actual-body}, and assume the directed packing condition
\eqref{eq:H}. Then
\begin{equation}\label{eq:reduced-signed}
 \E_{Y,\varepsilon}\max_{t\in T}
       \left|\sum_{i\in A_t}\varepsilon_iY_i\right|
       \ge\frac{c_B}{64}\,u.
\end{equation}
\end{theorem}

\begin{proof}
The case $m=0$ is incompatible with \eqref{eq:H}, so $m\ge1$ and
$p\ge64$. By \cref{lem:tail}, the bodies meet all assumptions of
\cref{thm:encoder}. Apply that theorem and \cref{cor:selection} to get
$T'\subset T$ with $|T'|\ge e^{p/2}$ and fixed witnesses $a_t$ such that
\begin{equation}\label{eq:reduced-selected}
 \Prob(Y\ge a_t)\ge e^{-p},\qquad
 a_t(A_t\setminus A_s)\ge u/8\quad(t\ne s\text{ in }T').
\end{equation}
Put $b_t=a_t/4$. By \eqref{eq:tail-shrink},
\[
 q_t=\Prob(Y\ge b_t)\ge e^{-p/4},\qquad
 \Lambda=\sum_{t\in T'}q_t\ge e^{p/4},\qquad
 \beta=u/32.
\]
Since
\[
 2L_m\le4e^{(2+\log2)m}
       \le4e^{(2+\log2)p/64}\le e^{p/4}\le\Lambda
       \qquad(p\ge64),
\]
\cref{lem:signed-count} gives the bound $c_Bu/64$ on $T'$.
The supremum over $T$ is at least as large.
\end{proof}

The positive analogue follows from \cref{prop:positive-count}, with
an even smaller counting threshold. The signed theorem is used for the
final reduction, so positivity is not imposed on the original process.

\section{Proof of the main theorem}\label{sec:main-proof}

We now verify that the classical reduction gives all hypotheses of
\cref{thm:reduced}, with constants that allow the approximation error
to be absorbed. The distinction between the original moment parameter
$P$ and the entropy parameter $p$ is kept explicit.

\begin{proof}[Proof of \cref{thm:main}]
Write $W=\W_X(T)$. Unconditionality gives $\E X=0$ and the sign
representation \eqref{eq:sign-representation}. Delete zero coordinates
and make a positive diagonal change to isotropic position if needed;
the corresponding inverse change of labels preserves the process.

Fix in advance
\begin{equation}\label{eq:fixed-constants}
 \eta=\frac1{256},\qquad
 c_0=\frac{c_B}{1024C_w},\qquad
 \epsilon=\frac{c_0}{2}.
\end{equation}
Increase $P_0=P_0(\eta,\epsilon)$ if necessary so that $P_0\ge256$.
Suppose first that $P\ge P_0$. Apply \cref{input:reduction}.
If \eqref{eq:easy-reduction-branch} holds, there is nothing further
to prove. Otherwise use its binary family, and set
\begin{equation}\label{eq:entropy-vs-moment}
 N=\lceil e^{P/4}\rceil,\qquad p=\log N,
 \qquad \frac P4\le p\le\frac P2.
\end{equation}
The last upper bound holds in the present large-$P$ range. Set
$Y_i=r_i|X_i|$. This is a nonnegative log-concave NA vector, and
\[
 M_{\rm bin}=\E_{Y,\varepsilon}\max_t
                      \left|\sum_{i\in A_t}\varepsilon_iY_i\right|.
\]
By \eqref{eq:reduction-support},
\begin{equation}\label{eq:thin-final}
 m:=\max_t|A_t|\le\frac P{256}\le\frac p{64}.
\end{equation}

Fix distinct labels $s,t$ and let $J=A_t\triangle A_s$.
Then $|J|\le2m\le P$. By unconditionality, the linear form in
\eqref{eq:reduction-packing} has the same law as
$\sum_{i\in J}r_iX_i$. Hence \cref{input:witness} gives a vector
$v\ge0$, supported on $J$, with
\begin{equation}\label{eq:pair-witness-P}
 \Prob(Y\ge v)\ge e^{-P},\qquad \norm{v}_1\ge\frac{a}{2C_w}.
\end{equation}
At least one directed difference, say $D=A_t\setminus A_s$, carries
mass at least $a/(4C_w)$. Restrict $v$ to $D$; this only increases
the probability of its upper-orthant event. Next multiply it by
$\theta=p/P\in[1/4,1/2]$. By \cref{lem:tail}, the resulting vector
$w$ satisfies
\begin{equation}\label{eq:pair-witness-p}
 \supp w\subset D,\qquad \Prob(Y\ge w)\ge e^{-p},\qquad
 \norm{w}_1\ge\frac{a}{16C_w}.
\end{equation}
Put
\begin{equation}\label{eq:u-final}
                     u=\frac{a}{16C_w}.
\end{equation}
If necessary scale $w$ down once more to have mass exactly $u$.
It then belongs to $K_{A_t}(p,u)$ and is supported on $D$.
The other orientation of the pair is treated identically.
Thus \eqref{eq:H} holds for the actual joint-tail bodies.

Apply \cref{thm:reduced} and then \eqref{eq:reduction-transfer}:
\[
 M_{\rm bin}\ge\frac{c_Bu}{64}=c_0a,
 \qquad
 2W\ge M_{\rm bin}-\epsilon a\ge\frac{c_0a}{2}.
\]
Consequently $W\ge c_0a/4$ in this branch. The constants in
\eqref{eq:fixed-constants} were chosen before applying the reduction;
there is no dependence of the reduced lower bound on a subsequently
chosen sparsity parameter.

Finally, let $1\le P<P_0$. Select any distinct $s,t\in T$ and set
$V=\ip{t-s}{X}$. By \eqref{eq:scalar-growth},
\[
 \E|V|\ge\frac{a}{C_{\rm reg}P_0}.
\]
Symmetry and centering give
\[
 2W=\E\left(\max_{t\in T}\ip{t}{X}
                  -\min_{t\in T}\ip{t}{X}\right)
       \ge\E|V|.
\]
Taking
\[
 c=\min\left\{c_{\eta,\epsilon},\frac{c_0}{4},
                  \frac1{2C_{\rm reg}P_0}\right\}
\]
proves \eqref{eq:main-smp} for every $P\ge1$.
\end{proof}

\section{Two consequences}\label{sec:consequences}

\subsection{Products of Orlicz-based measures}

The first consequence gives an explicit class of dependent measures to
which the theorem applies. A \emph{Young function} is a lower
semicontinuous, convex nondecreasing function
$\phi:[0,\infty)\to[0,\infty]$ with $\phi(0)=0$, not identically
zero and finite at some positive point.

\begin{corollary}\label{cor:orlicz}
Let $I_1,\ldots,I_k$ be a partition of $[d]$. For each $i$ let $\phi_i$
be a Young function, and for each $j$ let
$h_j:[0,\infty)\to[0,\infty)$ be log-concave and nonincreasing.
Assume that
\begin{equation}\label{eq:orlicz-density}
 f(x)=\frac1Z\prod_{j=1}^k
       h_j\!\left(\sum_{i\in I_j}\phi_i(|x_i|)\right)
\end{equation}
is a probability density, with $0<Z<\infty$ and $h_j(\infty)=0$.
Then a vector with density $f$ satisfies the conclusion of
\cref{thm:main}, with the same universal constant, independent of
$k$, the block dimensions, and the defining functions. In particular,
the conclusion holds for the uniform measure on every Cartesian
product of generalized Orlicz balls.
\end{corollary}

\begin{proof}
Each factor in \eqref{eq:orlicz-density} is unconditional and
log-concave. Indeed, $-\log h_j$ is convex and nondecreasing on its
effective domain, so its composition with
$\sum_{i\in I_j}\phi_i(|x_i|)$ is convex.
Wojtaszczyk's theorem~\cite[Theorem~1.3]{Wojtaszczyk} gives NA of the
magnitudes in each block. The blocks are independent because the
density factors.

For completeness, independent unions of NA families are NA. For two
independent blocks $U,V$ and increasing functions $F,G$ depending on
disjoint coordinate sets, fix $V$ and apply NA in $U$. This bounds
$\E_U(FG)$ by $(\E_U F)(\E_U G)$. The two expectations are increasing
functions of disjoint coordinate sets of $V$, so NA in $V$ gives
$\E(FG)\le\E F\,\E G$. Iteration proves the assertion for all
blocks. Thus \cref{thm:main} applies.

Taking $h_j=\one_{[0,R_j]}$ gives the uniform law on
\[
 \prod_{j=1}^k
 \left\{x_{I_j}:\sum_{i\in I_j}\phi_i(|x_i|)\le R_j\right\},
\]
whenever these sets are convex bodies.
\end{proof}

For example, one may take $h_j(s)=\exp(-V_j(s))$ with $V_j$ convex
and nondecreasing, allowing the value $+\infty$ to impose a hard
constraint. Weighted $\ell_q$ balls, $q\ge1$, correspond to
$\phi_i(s)=w_i s^q$ with $w_i>0$. Different blocks may have different
Young functions and different profiles $h_j$. The conclusion is
therefore uniform over products with several separate Orlicz constraints.

\Needspace{28\baselineskip}
\subsection{Coverings in the moment metric}

For a bounded nonempty set $T\subset\R^d$, let $N(T,d_p,r)$ denote
the least number of closed $d_p$-balls of radius $r$, with centers in
$T$, needed to cover $T$. This definition also applies when $d_p$
is a pseudometric. Write $\W_X(T)=\E\sup_{t\in T}\ip{t}{X}$;
in finite dimension this is finite for bounded $T$.

\begin{corollary}\label{cor:covering}
Let $X$ satisfy the assumptions of \cref{thm:main}. There is a universal
$C>0$ such that, for every bounded nonempty $T\subset\R^d$ and every
$p\ge1$,
\begin{equation}\label{eq:moment-covering}
 N\bigl(T,d_p,C\W_X(T)\bigr)<e^p.
\end{equation}
In particular, suppose $X$ is nondegenerate and define its symmetric
$L_p$ centroid body by
\[
 h_{Z_p(X)}(t)=\norm{\ip{t}{X}}_p.
\]
For every origin-symmetric convex body $K\subset\R^d$,
\begin{equation}\label{eq:centroid-covering}
 N\bigl(K,C\E h_K(X)\,Z_p(X)^\circ\bigr)<e^p,
\end{equation}
where $N(K,L)$ is the least number of translates of $L$ covering $K$,
and $L^\circ$ denotes the polar body. Consequently, for every norm
on $\R^d$ with dual unit ball $B^*$,
\begin{equation}\label{eq:norm-covering}
 N\bigl(B^*,C\E\norm{X}\,Z_p(X)^\circ\bigr)<e^p.
\end{equation}
\end{corollary}

\begin{proof}
Let $c$ be the constant in \cref{thm:main}, and take $C=2/c$.
First suppose $W:=\W_X(T)>0$ and set $r=CW$. If a subset
$S\subset T$ of size $\lceil e^p\rceil$ had all its pairwise
$d_p$-distances greater than $r$, then \cref{thm:main} would give
\[
 W\ge\W_X(S)\ge cr=2W,
\]
a contradiction. A greedy choice of points at distance greater than
$r$ from all previously chosen points must therefore stop after
fewer than $e^p$ points. At that time the closed balls of radius $r$
cover $T$, proving \eqref{eq:moment-covering}.

If $W=0$, centering and symmetry imply, for every $s,t\in T$,
\[
 \E|\ip{t-s}{X}|
   =2\E\max\{\ip{t}{X},\ip{s}{X}\}\le2W=0.
\]
Thus $d_p(s,t)=0$ and one ball of radius zero suffices.
Finally, the unit ball of $t\mapsto\norm{\ip{t}{X}}_p$ is exactly
$Z_p(X)^\circ$. Apply \eqref{eq:moment-covering} to $T=K$ to obtain
\eqref{eq:centroid-covering}. Taking $K=B^*$ and using
$h_{B^*}(x)=\norm{x}$ gives \eqref{eq:norm-covering}.
\end{proof}

The covering estimate is a geometric formulation of the minoration:
at each moment order $p$, fewer than $e^p$ translates of the polar
centroid body suffice at a scale fixed by the expected supremum.
It involves the centroid body of the dependent law itself.

\section{Remarks on the argument}\label{sec:remarks}

The deterministic and probabilistic steps use different hypotheses.
Convexity and downward closure of the witness bodies suffice for
\cref{thm:encoder}; no dependence assumption is used there.
Log-concavity supplies these bodies through joint upper-orthant
probabilities and allows the selected witnesses to be shrunk at a
controlled probability cost. Unconditionality supplies independent
signs and the sparse reduction. Negative association is used in
\eqref{eq:NA-G}, between a fixed witness event on $A_t$ and a fixed
increasing function of the coordinates in $A_t^c$.

The construction does not require a single threshold vector whose
restrictions serve every label. It produces one threshold vector per
retained label, fixed against every other retained label, by averaging
complete encoders under one common law. This simultaneous construction
avoids an iteration of support selections. The expected-count estimate
\eqref{eq:count-master} then controls the overlapping witness events
collectively, retaining the random signs. The proof uses neither
stability of NA under conditioning nor a dimension-free concentration
theorem. Removing NA from \cref{thm:main} would require a replacement
for this counting estimate.

\section*{Funding}
The first author was supported by NCN Grant UMO-2022/47/B/ST1/02114.

\end{document}